\documentclass[platex]{amsart}
\usepackage{amssymb}
\usepackage{braket}
\usepackage{mathrsfs}
\usepackage{amsrefs}
\usepackage{ifthen}
\usepackage{graphicx}
\usepackage{tikz}
\usetikzlibrary{patterns}
\usepackage{comment} 
\usepackage{mleftright}
\usepackage{cleveref}
\usepackage[all]{xy}

\theoremstyle{plain}
\newtheorem{theo}{Theorem}[section]
\crefname{theo}{Theorem}{Theorems}
\Crefname{theo}{Theorem}{Theorems}

\crefname{prop}{Proposition}{Propositions}
\Crefname{prop}{Proposition}{Propositions}
\newtheorem{lem}[theo]{Lemma}
\crefname{lem}{Lemma}{Lemmas}
\Crefname{lem}{Lemma}{Lemmas}
\newtheorem{cor}[theo]{Corollary}
\crefname{cor}{Corollary}{Corollaries}
\Crefname{cor}{Corollary}{Corollaries}

\crefname{claim}{Claim}{Claims}
\Crefname{claim}{Claim}{Claims}

\crefname{property}{Property}{Properties}
\Crefname{property}{Property}{Properties}

\crefname{problem}{Problem}{Problems}
\Crefname{problem}{Problem}{Problems}

\theoremstyle{definition}

\crefname{defi}{Definition}{Definitions}
\Crefname{defi}{Definition}{Definitions}

\crefname{notation}{Notation}{Notations}
\Crefname{notation}{Notation}{Notations}

\crefname{convention}{Convention}{Conventions}
\Crefname{convention}{Convention}{Conventions}

\crefname{cond}{Condition}{Conditions}
\Crefname{cond}{Condition}{Conditions}

\crefname{assum}{Assumption}{Assumptions}
\Crefname{assum}{Assumption}{Assumptions}

\theoremstyle{remark}
\newtheorem{rem}[theo]{Remark}
\crefname{rem}{Remark}{Remarks}
\Crefname{rem}{Remark}{Remarks}

\crefname{ex}{Example}{Examples}
\Crefname{ex}{Example}{Examples}

\crefname{section}{Section}{Sections}
\Crefname{section}{Section}{Sections}
\crefname{subsection}{Subsection}{Subsections}
\Crefname{subsection}{Subsection}{Subsections}
\crefname{figure}{Figure}{Figures}
\Crefname{figure}{Figure}{Figures}

\newtheorem*{acknowledgement}{Acknowledgement}

\newcommand{\Z}{\mathbb{Z}}

\newcommand{\R}{\mathbb{R}}
\newcommand{\C}{\mathbb{C}}

\newcommand{\CP}{\mathbb{CP}}
\newcommand{\RP}{\mathbb{RP}}

\newcommand{\gauge}{{\mathscr G}}

\newcommand{\calK}{{\mathcal K}}
\newcommand{\calV}{{\mathcal V}}

\newcommand{\calH}{{\mathcal H}}
\newcommand{\calW}{{\mathcal W}}

\newcommand{\fraks}{\mathfrak{s}}

\newcommand{\Diff}{\mathrm{Diff}}
\newcommand{\Homeo}{\mathrm{Homeo}}
\newcommand{\Aut}{\mathrm{Aut}}

\newcommand{\Si}{\Sigma}

\newcommand{\inc}{\hookrightarrow}

\newcommand{\im}{\mathop{\mathrm{Im}}\nolimits}
\newcommand{\rank}{\mathop{\mathrm{rank}}\nolimits}
\newcommand{\Hom}{\mathop{\mathrm{Hom}}\nolimits}

\newcommand{\Spin}{\operatorname{\mathrm {Spin}}}
\newcommand{\Spincm}{\Spin^{c_-}}
\newcommand{\OO}{\operatorname{\rm O}}
\newcommand{\Pin}{\mathrm{Pin}}
\newcommand{\fake}{\mathrm{fake}}

\title[Families of smooth 4-manifolds and $\Pin^{-}(2)$-monopole]{Constraints on families of smooth 4-manifolds from $\Pin^{-}(2)$-monopole}

\author{Hokuto Konno}
\address{2-1 Hirosawa, Wako, Saitama 351-0198, Japan}
\email{hokuto.konno@riken.jp}

\author{Nobuhiro Nakamura}
\address{Department of Mathematics, Osaka Medical College, 2-7 Daigaku-machi, Takatsuki City, Osaka, 569-8686, Japan}
\email{mat002@osaka-med.ac.jp}

\begin{document}

\maketitle

\begin{abstract}
Using the Seiberg--Witten monopole equations, Baraglia recently proved that for most of simply-connected closed smooth 4-manifolds $X$, the inclusions $\mathrm{Diff}(X) \hookrightarrow \mathrm{Homeo}(X)$ are not weak homotopy equivalences.
In this paper, we generalize Baraglia's result using the $\Pin^{-}(2)$-monopole equations instead.
We also give new examples of $4$-manifolds $X$ for which $\pi_{0}(\Diff(X)) \to \pi_{0}(\Homeo(X))$ are not surjections.
\end{abstract}

\section{Introduction}

T. Kato and the authors~\cite{KKN} recently made use of Seiberg--Witten theory for families in order to detect non-smoothable topological families of $4$-manifolds.
This argument extracts some homotopical difference between the homeomorphism groups and the diffeomorphism groups of some class of $4$-manifolds.
Soon after \cite{KKN}, using Seiberg--Witten theory for families in a different manner, D. Baraglia~\cite{B} extensively generalized the result in \cite{KKN} on comparisons between the homeomorphism and diffeomorphism groups of $4$-manifolds:
he proved in \cite[Corollary~1.9]{B} that for every closed, oriented, simply-connected, smooth, and indefinite $4$-manifold $M$ with $|\sigma(M)|>8$, the inclusion $\Diff(M) \inc \Homeo(M)$ is not a weak homotopy equivalence.
Here $\sigma(M)$ denotes the signature of $M$, and $\Diff(M)$ and $\Homeo(M)$ denote the groups of diffeomorphisms and homeomorphisms respectively.
The proof of this result by Baraglia is based on a finite-dimensional approximation of the families Seiberg--Witten monopole map.
The purpose of this paper is to give analogues of arguments in \cite{B} by Baraglia for the $\Pin^{-}(2)$-monopole equations introduced in \cite{N1}, and to make use of the $\Pin^{-}(2)$-monopole analogues to generalize the above result by Baraglia on comparison between homeomorphism and diffeomorphism groups as follows:

\begin{theo}
\label{main appl 1}
Let $X$ be a smooth $4$-manifold which is homeomorphic to a $4$-manifold of the form
\begin{align}
\label{eq: MN exporession}
M \#_{i=1}^{p}(S^{1} \times Y_{i})\#_{j=1}^{q} (S^{2} \times \Si_{j}),
\end{align}
where
\begin{itemize}
\item $M$ is a simply-connected, closed, oriented, smooth, and indefinite $4$-manifold with $|\sigma(M)| > 8$;
\item $Y_{i}$ is an oriented closed $3$-manifold, and $\Si_{j}$ is an oriented closed $2$-manifold of positive genus; and
\item $p$ and $q$ are non-negative integers, where we interpret $\#_{i=1}^{p}(S^{1} \times Y_{i})$ as $S^{4}$ for $p=0$, and similarly for $q=0$.
\end{itemize}
Set $n = \min\{b_{+}(M), b_{-}(M)\}$.
If we fix a homeomorphism between $X$ and a $4$-manifold of the form \eqref{eq: MN exporession}, then:
\begin{itemize}
\item If $M$ is non-spin, there exists a non-smoothable $\Homeo(X)$-bundle
\[
X \to E \to T^{n}.
\]
\item If $M$ is spin, there exists a non-smoothable $\Homeo(X)$-bundle
\[
X \to E \to T^{n-1}.
\]
\end{itemize}
\end{theo}

Here $b_{+}(M)$ denotes the maximal dimension of positive-definite subspaces of $H^{2}(M;\R)$ with respect to the intersection form, and $b_{-}(M)=b_{2}(M)-b_{+}(M)$.
We say that a $\Homeo(X)$-bundle $E$ is {\it non-smoothable} if $E$ does not admit a reduction of structure to $\Diff(X)$.

By standard obstruction theory, we have:

\begin{cor}
\label{cor homotpy groups}
Let $X$ be a smooth $4$-manifold which is homeomorphic to a $4$-manifold of the form
\[
M \#_{i=1}^{p}(S^{1} \times Y_{i})\#_{j=1}^{q} (S^{2} \times \Si_{j}),
\]
where
\begin{itemize}
\item $M$ is a simply-connected, closed, oriented, smooth, and indefinite $4$-manifold with $|\sigma(M)| > 8$;
\item $Y_{i}$ is an oriented closed $3$-manifold, and $\Si_{j}$ is an oriented closed $2$-manifold of positive genus; and
\item $p$ and $q$ are non-negative integers.
\end{itemize}
Then the inclusion
\[
\Diff(X) \inc \Homeo(X)
\]
is not a weak homotopy equivalence.

More precisely, if we fix a homeomorphism between $X$ and a $4$-manifold of the form \eqref{eq: MN exporession}, then:
\begin{itemize}
\item If $M$ is non-spin,
\[
\pi_{k}(\Diff(X)) \to \pi_{k}(\Homeo(X))
\]
is not an isomorphism for some $k \leq \min\{b_{+}(M), b_{-}(M)\} -1$.
\item If $M$ is spin,
\[
\pi_{k}(\Diff(X)) \to \pi_{k}(\Homeo(X))
\]
is not an isomorphism for some $k \leq \min\{b_{+}(M), b_{-}(M)\} - 2$.
\end{itemize}
\end{cor}

\begin{rem}
\label{rem: comparison with B}
Here we compare \cref{main appl 1} and \cref{cor homotpy groups} with Baraglia's argument given in \cite{B}:
\begin{enumerate}
\item The case that $p=q=0$ follows from an argument based on \cite[Theorem~1.1]{B}.
\item The case that $p=0$, $q \leq 2$, and $M$ is spin follows from an argument based on \cite[Theorem~1.2]{B}.
\end{enumerate}
\end{rem}


Instead of a simply-connected $4$-manifold in $M$ in \cref{main appl 1} and \cref{cor homotpy groups}, we may also consider non-simply-connected $4$-manifolds whose homeomorphism types can be understood very well.
We give such an example using Enriques surfaces:

\begin{theo}
\label{main appl 2}
Let $X$ be a smooth $4$-manifold which is homeomorphic to a $4$-manifold of the form
\begin{align*}
mS\#M\#_{i=1}^{p}(S^{1} \times Y_{i})\#_{j=1}^{q} (S^{2} \times \Si_{j}),
\end{align*}
where
\begin{itemize}
\item $S$ is an Enriques surface and $M$ is a standard simply-connected smooth 4-manifold.
Here $M$ is called standard if $M$ is obtained as the connected sum of finitely many (possibly zero) copies of $\CP^{2}$, $-\CP^{2}$, $S^{2} \times S^{2}$, $K3$, and $-K3$.
If $M$ is spin, we assume that $\sigma(M) \leq 0$;
\item $Y_{i}$ is an oriented closed $3$-manifold, and $\Si_{j}$ is an oriented closed $2$-manifold of positive genus; and
\item $m$ is a positive integer, and $p$ and $q$ are non-negative integers, where we interpret $\#_{i=1}^{p}(S^{1} \times Y_{i})$ as $S^{4}$ for $p=0$, and similarly for $q=0$.
\end{itemize}
Set $n=b_{+}(M)+m$.
Then there exists a non-smoothable $\Homeo(X)$-bundle
\[
X \to E \to T^{n}.
\]
\end{theo}

\begin{cor}
\label{cor homotpy groups2}
Let $X$ be a smooth $4$-manifold which is homeomorphic to a $4$-manifold of the form
\begin{align*}
mS\#M\#_{i=1}^{p}(S^{1} \times Y_{i})\#_{j=1}^{q} (S^{2} \times \Si_{j}),
\end{align*}
where
\begin{itemize}
\item $S$ is an Enriques surface and $M$ is a standard simply-connected smooth 4-manifold.
If $M$ is spin, we assume that $\sigma(M) \leq 0$;
\item $Y_{i}$ is an oriented closed $3$-manifold, and $\Si_{j}$ is an oriented closed $2$-manifold of positive genus; and
\item $m$ is a positive integer, and $p$ and $q$ are non-negative integers.
\end{itemize}
Then the inclusion
\[
\Diff(X) \inc \Homeo(X)
\]
is not a weak homotopy equivalence.
More precisely,
\[
\pi_{k}(\Diff(X)) \to \pi_{k}(\Homeo(X))
\]
is not an isomorphism for some $k \leq b_{+}(M)+m-1$.
\end{cor}

As a more specific corollary of \cref{main appl 2} than \cref{cor homotpy groups2}, we may give new examples of $4$-manifolds $X$ for which $\pi_{0}(\Diff(X)) \to \pi_{0}(\Homeo(X))$ are not surjections:

\begin{cor}
\label{cor homotpy groups3}
Let $X$ be a smooth $4$-manifold which is homeomorphic to a $4$-manifold of the form
\begin{align*}
S\#k(-\CP^{2})\#_{i=1}^{p}(S^{1} \times Y_{i})\#_{j=1}^{q} (S^{2} \times \Si_{j}),
\end{align*}
where
\begin{itemize}
\item $S$ is an Enriques surface, $Y_{i}$ is an oriented closed $3$-manifold, and $\Si_{j}$ is an oriented closed $2$-manifold of positive genus; and
\item $k$, $p$ and $q$ are non-negative integers.
\end{itemize}
Then 
\[
\pi_{0}(\Diff(X)) \to \pi_{0}(\Homeo(X))
\]
is not a surjection.
Namely, there exists a self-homeomorphism of $X$ which is not topologically isotopic to any self-diffeomorphism of $X$.
\end{cor}

\begin{rem}
The case in \cref{main appl 2} and \cref{cor homotpy groups2,cor homotpy groups3} that $p=q = 0$ can be deduced also from an argument using \cite[Theorems~1.1]{B}.
\end{rem}

The first example of $4$-manifolds $X$ for which $\pi_{0}(\Diff(X)) \to \pi_{0}(\Homeo(X))$ are not surjections is a $K3$ surface, proven by Donaldson~\cite{D}.
One may check the same statement holds also for any homotopy $K3$ surface using the Seiberg--Witten invariants and a result by Morgan and Szab\'{o}~\cite{MS}.
We note that examples of $4$-manifolds $X$ for which $\pi_{0}(\Diff(X)) \to \pi_{0}(\Homeo(X))$ are not {\it injections} are known a little more: the first example was given by Ruberman~\cite{R}, and later additional examples were given by Baraglia and the first author~\cite{BK0}, and by Kronheimer and Mrowka~\cite{KM} recently.

The paper is organized as follows.
In \cref{section:pin2-monopole} we recall some basics of $\Pin^-(2)$-monopole theory and describe a finite-dimensional approximation of the families $\Pin^-(2)$-monopole map.
In \cref{section: Main theorems} we give constraints on smooth families of $4$-manifold using a finite-dimensional approximation of a families $\Pin^-(2)$-monopole map.
Those constraints are analogues of some constraints by Baraglia~\cite{B} obtained from the families Seiberg-Witten monopole map.
In \cref{section: Proof of main appl 1} we give the proofs of  \cref{main appl 1,main appl 2}:
we shall construct concrete topological families of $4$-manifolds and show the non-smoothability of them using the constraints obtained in \cref{section: Main theorems}.

\begin{acknowledgement}
The first author was partially supported by JSPS KAKENHI Grant Numbers 17H06461 and 19K23412.
The second author was supported by JSPS KAKENHI Grant Number 19K03506.
\end{acknowledgement}

\section{$\Pin^-(2)$-monopole maps for families}
\label{section:pin2-monopole}

  First, we briefly review $\Pin^-(2)$-monopole theory.
 For a thorough treatment, readers are referred to \cites{N1,N2}.
  
  Let $X$ be an oriented, closed, connected, and smooth $4$-manifold.
  Fix a Riemannian metric $g$ on $X$.
  Let $\widetilde{X} \to X$ be an unbranched double cover, and let $\ell = \widetilde{X} \times_{\{\pm 1\}} \Z$, the associated local system with coefficient group $\Z$.
We always assume that $\widetilde{X} \to X$ is nontrivial throughout this paper.
  Let $\ell_{\R} = \ell\otimes\R$ and $i\ell_{\R}=\ell\otimes\sqrt{-1}\R$.
   Set $b_j^{\ell}(X) = \rank H^j(X;\ell)$ for $j \geq 0$, and set $b_{+}^{\ell}(X) = \rank H^{+}(X;\ell)$, where $H^{+}(X;\ell)$ denotes a maximal-dimensional positive-definite subspace of $H^{2}(X;\ell)$ with respect to the intersection form of $X$.
  Define the Lie groups $\Pin^-(2)$, and $\Spincm(4)$ by $\Pin^{-}(2) = \mathrm{U}(1) \cup j\mathrm{U}(1) \subset \mathrm{Sp}(1)$ and $\Spincm(4) = \Spin(4) \times_{\{\pm 1\}} \Pin^-(2)$.
  Note that $\Spincm(4)/\Spin^{c}(4) \cong \{\pm1\}$ and $\Spincm(4)/\Pin^{-}(2) \cong \mathrm{SO}(4)$.
    A {\it $\Spincm$-structure} on $\widetilde{X} \to X$ is defined as a triple $\mathfrak{s} = (P, \sigma, \tau)$, where
    \begin{itemize}
      \item $P$ is a principal $\Spincm(4)$-bundle over $X$,
      \item $\sigma : \widetilde{X} \to P/\Spin^{c}(4)$ is an isomorphism of $\{\pm1\}$-bundles, and
      \item $\tau : {\rm Fr}(X) \to P/\Pin^{-}(2)$ is an isomorphism of $SO(4)$-bundles, where ${\rm Fr}(X)$ denotes the frame bundle of $X$.
    \end{itemize}
  The associated $\mathrm{O}(2)$-bundle $L = P/\Spin(4)$ is called the {\it characteristic bundle} of a $\Spin^{c-}$-structure $\mathfrak{s} = (P,\sigma,\tau)$.
  We denote the $\ell$-coefficient Euler class of $L$ by $\tilde{c}_1(\mathfrak{s}) \in H^2(X; \ell)$.
 
  Some notions associated to $\Spincm$-structures are very similar to those of $\Spin^c$-structures:
  a $\Spin^{c-}$-structure $\mathfrak{s}$ on $\widetilde{X} \to X$ gives rise to the positive and negative spinor bundles $S^{\pm}$ over $X$ and the Clifford multiplication $\rho : \Omega^1(X; i\ell_{\R}) \to \Hom(S^+,S^-)$.
  An $\mathrm{O}(2)$-connection $A$ on $L$ induces the Dirac operator $D_A : \Gamma(S^+) \to \Gamma(S^-)$. 
  Note that the curvature $F_A^+$ is an element of $\Omega^+(X;i\ell_{\R})$.
   We denote by $q \colon S^+ \to \Omega^+(X; i\ell_{\R})$ the canonical real quadratic map.
The {\it $\Pin^-(2)$-monopole equations} is defined by
  \begin{equation}\label{eq:pin2-monopole}
	\left\{
	\begin{aligned}
		D_A \phi &= 0, \\
		\frac{1}{2} F_A^+ &= q(\phi)
	\end{aligned}
	\right.
  \end{equation}
  for $\OO(2)$-connections $A$ on $L$ and positive spinors $\phi \in \Gamma(S^+)$.
  The  equations \eqref{eq:pin2-monopole} are equivariant under the action of the gauge group $\gauge$, which is defined by $\gauge = \Gamma(\widetilde{X} \times_{\{\pm 1\}} \mathrm{U}(1))$. 

 Choose a reference $\OO(2)$-connection $A_0$ on $L$.
 The {\it $\Pin^-(2)$-monopole map}
  \begin{gather*} m\colon\Omega^1(X;i\ell_{\R})\oplus\Gamma(S^+)\to H^1(X;\ell_{\R})\oplus(\Omega^0\oplus\Omega^+)(X;i\ell_{\R})\oplus\Gamma(S^-)
 \end{gather*}
 is defined by 
  \begin{gather*}
 m(a,\phi)=(h(a), d^*a,d^+a-q(\phi),D_{A_0+a}\phi),
  \end{gather*}
 where 
 $h(a)$ denotes the harmonic part of the $1$-form $a$.
 The  map $m$ is decomposed into the sum $m=l+c$, where $l$ is the linear map given by $l=(d^*,d^+, D_{A_0})$, and $c$ is the quadratic part given by $c(a,\phi)=(0,-q(\phi),\frac12\rho(a)\phi)$.
 As well as usual Seiberg--Witten theory, we consider the Sobolev completions of the domain and the target of $m$.
Choose $k\geq 4$. Let $\calV:= L^2_k(\Omega^1(X;i\ell_{\R})\oplus\Gamma(S^+))$ and $\calW:=L^2_{k-1}((\Omega^0\oplus\Omega^+)(X;i\ell_{\R})\oplus\Gamma(S^-))$.
 Then $m$ is extended to a smooth map $m\colon \calV\to H^1(X;\ell_{\R})\oplus\calW$.
 The linear part $l$ is a Fredholm map of index
 \[
 \frac14(\tilde{c}_1(\mathfrak{s})^2-\sigma(X)) + b_1^{\ell}(X) -b_+^{\ell}(X),
 \]
 and $c$ is a non-linear compact map.
Note that $b_0^{\ell}(X)=0$ if $\ell$ is non-trivial.
 
 We take the $L^2_{k+1}$-completion of the gauge group $\gauge$, denoted by the same symbol $\gauge$ to simplify the notation.
 Then the $\gauge$-action is smooth.
 The space 
\[
 \ker (d^*\colon L^2_k(\Omega^{1}(X;i\ell_\R))\to L^2_{k-1}(\Omega^{0}(X;i\ell_\R)))
\]
is a global slice for the $\gauge$-action at $(0,0)$, and we have
 \[
 m^{-1}(0) = \{\text{solutions to \eqref{eq:pin2-monopole}}\}\cap \ker d^*.
 \]
 The slice $\ker d^*$ still has a remaining gauge symmetry. 
 Let $\calH$ be the kernel of the composition of the maps
  \[
   \xymatrix{
 L^2_{k+1}(\gauge) \ar[r]^-{d} & L^2_k(\Omega^{1}(X;i\ell_\R)) \ar[r]^-{d^{*}+d^{+}} & L^2_{k-1}((\Omega^0\oplus\Omega^+)(X;i\ell_{\R})).}
\]
  Then $m$ is $\calH$-equivariant, and we have
 \[
 m^{-1}(0)/\calH = \{\text{solutions to \eqref{eq:pin2-monopole}}\}/\gauge.
 \]
Note that
\[
H^1(X;\ell)=\Z_2\oplus \Z^{b_1^{\ell}}
\]
if $\ell$ is nontrivial.
Let $r\colon H^1(X;\ell)\to H^1(X;\ell_{\R})$ be the map induced from the natural map $\ell\to\ell_{\R}$ and set $\bar{H}:=\im r \cong \Z^{b_1^{\ell}}$.
Note the following exact sequence:
\begin{equation}\label{eq:harm-ex}
1\to\{\pm1\}\to \calH\to \bar{H}\to 0.
\end{equation}
Fixing a splitting of the above sequence, we have
\[
\calH \cong \{\pm1\}\times\bar{H}.
\]
\begin{rem}
A way of fixing a splitting of \eqref{eq:harm-ex} is as follows ({\it cf.} \cite[\S4.7]{N1}).
Choose a loop $\gamma$ in $X$ such that the restriction of $\ell$ to $\gamma$ is nontrivial.
Let $\calK_{\gamma}$ be the subgroup of $\gauge$ consisting of $u\in\gauge$ satisfying that $u|_{\gamma}$ is homotopic to the constant map with value $1$.
Then there is an exact sequence
\[
1\to\calK_{\gamma}\to\gauge\to\{\pm 1\}\to 1.
\]
From this we have 
\[
\calH\cap\calK_{\gamma} \cong \bar{H},
\]
and this gives a splitting of \eqref{eq:harm-ex}.
\end{rem}

Let $J:=H^1(X;\ell_{\R})/\bar{H}$. Then $J$ is a $b_1^{\ell}$-dimensional torus.
Dividing the harmonic projection 
\[
\varpi\colon\calV\to H^1(X;i\ell)\ ;\ (a,\phi)\mapsto h(a)
\]
by $\bar{H}$, we obtain a Hilbert bundle $\bar{\calV}=\calV/\bar{H}\to J$.
Then dividing the map $m$ by $\bar{H}$, we obtain a fiber-preserving $\{\pm 1\}$-equivariant map $\bar{m}$:
\begin{equation}\label{eq:m-bar}
\xymatrix{
\bar{\calV}\ar[r]^-{\bar{m}} \ar[d]^{\varpi}&J\times \calW\ar[d]\\
J\ar@{=}[r]&J.
}
\end{equation}
For our later purpose, there is no need for the whole of $\bar{m}$. 
What we need is  only the restriction $\bar{m}|_{\varpi^{-1}(0)}$ of $\bar{m}$ to the fiber over the origin of $J$.
The restriction  $\bar{m}|_{\varpi^{-1}(0)}$ is identified with the map $m_0$ defined by
\begin{equation}
\begin{gathered}\label{eq:monopole0}
\calV_0:= L^2_k(\im(d^*\colon\Omega^2(X;i\ell_{\R})\to\Omega^1(X;i\ell_{\R}))\oplus\Gamma(S^+)), \\
\calW_0:= L^2_{k-1}(\Omega^+(X;i\ell_{\R})\oplus\Gamma(S^-)),\\
m_0\colon \calV_0\to\calW_0\ ;\ 
(a,\phi)\mapsto( F_{A_0}+ d^+ a -q(a), D_{A_0+a}\phi).
\end{gathered}
\end{equation}

Let $B$ be a compact space.
Suppose  a smooth $\mathrm{Aut}(X,\mathfrak{s})$-bundle $(X,\mathfrak{s})\to E\to B$ is given. 
That is, $E$ is a smooth fiber bundle $E=\coprod_{b\in B} (X_b,\mathfrak{s}_b)$ with fiber a $\Spin^{c-}$ $4$-manifold such that there is an isomorphism $(X_b,\mathfrak{s}_b)\cong (X,\mathfrak{s})$ of $\Spin^{c-}$ $4$-manifolds for each $b$.
Let $\mathbb{L}=\coprod_{b\in B} L_b$ be the associated family of  $\OO(2)$-bundles where each $L_b$ is the characteristic $\OO(2)$-bundle of $(X_b,\mathfrak{s}_b)$.
Choose a family of Riemannian metrics $\{g_b\}_{b\in B}$ on $E$.
Then we have  an associated vector bundle
\[
\R^{b_{+}^{\ell}} \to H^+(E,\ell) \to B
\]
whose fiber over $b\in B$ is the space $H^+(X_b;\ell_b)$ of harmonic self-dual $2$-forms on $X_b$.
The isomorphism class of $H^{+}(E,\ell)$ is independent of the choice of the family of Riemannian metrics on $E$ since the Grassmannian of maximal-dimensional positive-definite subspaces of $H^{2}(X;\ell_{\R})$ is contractible.

Choose a family of reference $\OO(2)$-connections $\{A_b\}_{b\in B}$ on $\mathbb{L}$.
Then we can obtain a family of $m_0$ given in \eqref{eq:monopole0}, denoted by
\[
\mu_0\colon\tilde{\calV}\to\tilde{\calW},
\]
by parametrizing the previous argument over $B$.
Here $\tilde{\calV}$ and $\tilde{\calW}$ are the Hilbert bundles over $B$ with fibers $\calV_0$ and $\calW_0$ respectively,
 and $\mu_0$ is a fiber-preserving map whose restriction on each fiber is identified with the map $m_0$.

By taking a finite-dimensional approximation of $\mu_0$ \cite{F,BF,BK}, we obtain a $\{\pm 1\}$-equivariant proper map 
\[
f\colon V\to W
\]
which satisfies the following properties:
\begin{itemize}
\item $V$, $W$ are finite rank sub-bundles of $\tilde{\calV}$, $\tilde{\calW}$.
\item $V$ and $W$ are decomposed as $V=V_0\oplus V_1$ and $W=W_0\oplus W_1$. 
The group $\{\pm1\}$ acts on $V_0$ and $W_0$ trivially, and on $V_1$ and $W_1$ by fiberwise multiplication.
\item $f^{\{\pm 1\}} = f|_{V_0}\colon V_0\to W_0$ is a fiberwise linear incusion.
\item $W_0$ is isomorphic to $V_0\oplus H^+(E,\ell)$.
\item The index of the family of the Dirac operators, $\mathrm{ind} \{D_{A_b}\}$, is represented by $[V_1] -[W_1]$ in $K_{\{\pm 1\}}(B)$.  
\end{itemize}

When $\tilde{c}_1(\mathfrak{s})=0$, the $\Pin^-(2)$-monopole equations have a larger gauge symmetry given by  $\tilde{\gauge}=\Gamma(\tilde{X}\times_{\{\pm 1\}}\Pin^-(2))$ (\cite[\S4.3]{N1}).
Then the whole theory admits the $j$-action and the resulting finite-dimensional approximation $f\colon V\to W$ is equivariant under the action of the cyclic group $C_4$ of order $4$ generated by $j$.
In this case, $C_4$ acts on $V_0$ and $W_0$ by fiberwise multiplication of $\{\pm 1\}$ via the surjective homomorphism $C_4\to\{\pm1\}$, and on $V_1$ and $W_1$ by fiberwise multiplication of $j$.
Note that the $j$-action gives complex structures on $V_1$ and $W_1$.
\begin{rem}
As mentioned above, what we need for the proofs of our results is the family $\mu_0$ and its finite-dimensional approximation.
More generally, we can construct a parametrized family of the {\it total} monopole maps $\bar{m}$  of \eqref{eq:m-bar} once a family of splittings of \eqref{eq:harm-ex} is given.
We can obtain such a family of splittings if we can choose a family of loops $\{\gamma_b\}_{b\in B}$ such that $\ell|_{\gamma_b}$ is nontrivial.
In this case, the family of the monopole maps is parametrized by the total space of a bundle $K$ over $B$ with fiber $J$.  
\end{rem}

\section{Constraints from $\Pin^{-}(2)$-monopole}
\label{section: Main theorems}

As in \cref{section:pin2-monopole}, suppose that we have a smooth $\mathrm{Aut}(X,\mathfrak{s})$-bundle $(X,\mathfrak{s})\to E\to B$, where $B$ is a compact space.

The following theorem is a $\Pin^{-}(2)$-monopole analogue of a part of \cite[Theorem~1.1]{B} by Baraglia:

\begin{theo}
\label{main thm 1}
If $w_{b_+^{\ell}}(H^+(E,\ell))\neq 0$ in $H^{b_+^{\ell}}(B;\Z_2)$, then $\tilde{c}_1(\fraks)^2\leq \sigma(X)$ holds.	
\end{theo}
\begin{proof}
The proof is parallel to that of \cite[Theorem~1.1]{B}.
Throughout this proof, the coefficients of cohomology groups are supposed to be $\Z_2$.
Let $G=\{\pm 1\}$.
Note that the Borel cohomology $H_G^*(pt)$ is isomorphic to $ \Z_2[u]$ with $\deg u =1$. 
Since $G$ acts on the base space $B$ trivially, we have $H_G^*(B) \cong H^*(B)[u]$.
For a vector bundle $U$ over $B$, denote its disk bundle by $D(U)$, and the sphere bundle by $S(U)$.
Choosing a finite-dimensional approximation $f$ of $\mu_0$, we have the following commutative diagram,
\[
\xymatrix{
V=V_0\oplus V_1\ar[r]^{f} &W=W_0\oplus W_1\\
V_0\ar[r]^{f^G}\ar[u]^{\iota_0}&W_0\ar[u]^{\iota_1}.
}
\]
Note that the vertical arrows and $f^G$ are fiberwise linear inclusions.
We also have a relative version of the above diagram for the pairs $(D(V),S(V))$ etc.
Applying $H^*_G$-functor, we obtain
\begin{equation}\label{eq:comm}
\xymatrix{
H^*_G(D(V),S(V)) \ar[d]_{\iota_0^*}&H^*_G(D(W),S(W))\ar[l]_{f^*}\ar[d]_{\iota_1^*}\\
H^*_G(D(V_0),S(V_0))&H^*_G(D(W_0),S(W_0))\ar[l]_{(f^G)^*}.
}
\end{equation}
Note the following facts:
\begin{itemize}
\item The Thom isomorphisms, e.g., $H^*_G(D(V),S(V))\cong H^*_G(B)\tau_G(V)$, where $\tau_G(V)$ is the $G$-equivariant Thom class. 
\item $\iota_0^*\tau_G(V_0\oplus V_1)=e_G(V_1)\tau_G(V_0)$, where $e_G(V_1)$ is the $G$-equivariant Euler class.
Similarly, 
\begin{gather*}
\iota_1^*\tau_G(W_0\oplus W_1) = e_G(W_1)\tau_G(W_0),\\
(f^G)^*\tau_G(W_0) = e_G(H^+(E,\ell))\tau_G(V_0).
\end{gather*}
The last equation follows from that $W_0\cong V_0\oplus H^+(E,\ell)$ 
\item There exists a class $\alpha$ in $H_G^*(B)$ such that $f^*\tau_G(W)=\alpha\tau_G(V)$. 
The class $\alpha$ is called the  {\it cohomological degree} of $f$.
\end{itemize}
By the diagram \eqref{eq:comm}, we obtain the relation
\begin{equation}\label{eq:relation}
\alpha e_G(V_1)\tau_G(V_0) = e_G(H^+(E,\ell))e_G(W_1)\tau_G(V_0).
\end{equation}
Let $m=\rank_{\R} V_1$ and $n=\rank_{\R} W_1$.
Then 
\[
m-n = \operatorname{ind} D_{A_b} = \frac14(\tilde{c}_1(\fraks)^2 -\sigma(X)).
\]
The $G$-Euler classes of $V_1$ and $W_1$ are given by
\begin{align*}
e_G(V_1) =&w_m(V_1) + w_{m-1}(V_1)u+ \cdots+w_1(V_1)u^{m-1} + u^m,\\
e_G(W_1) =&w_n(W_1) + w_{n-1}(W_1)u+ \cdots+w_1(W_1)u^{n-1} + u^n.
\end{align*}
Since $G$ acts on $H^+(E,\ell)$ trivially, we have $e_G(H^+(E,\ell)) = w_{b_+^{\ell}}(H^+(E,\ell))$.
By \eqref{eq:relation}, $e_G(H^+(E^+,\ell))e_G(W_1)$ is divisible by $e_G(V_1)$.
If $e_G(H^+(E,\ell)) = w_{b_+^{\ell}}(H^+(E,\ell))\neq 0$, then $m-n\leq 0$.
Finally we obtain $\tilde{c}_1(\fraks)^2\leq \sigma(X)$.
\end{proof}
Using the relation \eqref{eq:relation}, we can obtain additional constraints on $V_1$ and $W_1$.
\begin{cor}\label{cor1}
For $i$ with $i> n-m$ , $w_i([W_1]-[V_1]) e(H^+(E,\ell))=0$.
\end{cor}
\begin{proof}
In $H^*(B)[u,u^{-1}]$, the equality \eqref{eq:relation} implies that
\[
\alpha = e_G(H^+(E^+,\ell))e_G(W_1)e_G(V_1)^{-1}.
 \]
 Since $\alpha$ is in $H^*(B)[u]$, the right-hand side has no terms of negative degree in $u$.
 \end{proof}

\begin{rem}
In the proofs of \cref{main thm 1} and \cref{cor1}, we used the $\Z_2$-coefficient Borel cohomology.
We can obtain similar constraints using the Borel cohomology with local coefficient $\Z_{w_1(H^+(E;\ell))}$.
In this case, the constraints are given in terms of  Chern classes of $V_1$ and $W_1$ with local coefficient.
\end{rem}

The following theorem is a $\Pin^{-}(2)$-monopole analogue of \cite[Theorem~1.2]{B}:

\begin{theo}
\label{main thm 2}
Suppose $\tilde{c}_1(\mathfrak{s})=0$ for the family $(X,\mathfrak{s})\to E\to B$.
If $w_{b_+^{\ell}}(H^+(E,\ell)) \neq 0$ or $w_{b_+^{\ell}-1}(H^+(E,\ell)) \neq 0$, then we have $\sigma(X) \geq 0$.
\end{theo}
\begin{proof}
Recall that a finite-dimensional approximation $f$ is $C_4$-equivariant , when $\tilde{c}_1(\mathfrak{s})=0$.
Let $G=C_4$. 
Also in this proof, the coefficients of cohomology groups are supposed to be $\Z_2$.
Then we have $H_G^*(pt)=\Z_2[u,v]/u^2$ with $\deg u=1$ and  $\deg v=2$.
The surjective homomorphism $G\to\{\pm 1\}$ induces the homomorphism 
\[
H^*_{\{\pm 1\}}(pt)=\Z_2[u]\to H^*_G(pt)=\Z_2[u,v]/u^2, \quad u\mapsto u.
\]
Regard $G$ as a subgroup of $S^1$ in an obvious way.
Then the inclusion  $G\hookrightarrow S^1$ induces the homomorphism
\[
H^*_{S^1}(pt)=\Z_2[v]\to H^*_G(pt)=\Z_2[u,v]/u^2, \quad v\mapsto v.
\] 
By an argument similar to the proof of \cref{main thm 1}, we obtain the relation \eqref{eq:relation} for some $\alpha\in H^*_G(B)$.
In this case, $V_1$ and $W_1$ are complex vector bundles.
Let $r:=\rank_{\C} V_1$  and $s:=\rank_{\C} W_1$. 
Then
\[
r-s = -\frac{\sigma(X)}8.
\]
The $G$-Euler classes are written as
\begin{align*}
e_G(V_1) =&c_r(V_1) + c_{r-1}(V_1)u+ \cdots+w_1(V_1)u^{r-1} + u^r,\\
e_G(W_1) =&c_s(W_1) + c_{s-1}(W_1)u+ \cdots+c_1(W_1)u^{s-1} + u^s,
\end{align*}
where $c_i$ are the mod-2-Chern classes.
If we regard $H=H^+(E,\ell)$ as a $\{\pm1\}$-equivariant bundle, then the $\{\pm 1\}$-Euler class of $H$ is given by
\[
e_{\{\pm 1\}} (H) = w_b(H)+w_{b-1}u+\cdots +w_1(H)u^{b-1} +u^b,
\]
where $b=b_+^{\ell}$.
Noticing $u^2=0$ in $H^*_G(B)$, we obtain
\[
e_G(H)=w_b(H)+w_{b-1}(H)u
\]
Then, under the assumption that $e_G(H)\neq 0$, the relation \eqref{eq:relation} implies that 
\[
-\frac{\sigma(X)}8=r-s\leq 0.
\]
This proves the \lcnamecref{main thm 2}.
\end{proof}

\begin{rem}
The proofs of \cite[Theorem~1.1]{B} and \cite[Theorem~1.2]{B} used $S^{1}$-symmetry and $\Pin(2)$-symmetry of the monopole maps respectively.
It would be worth noting that the above arguments of the proofs of \cref{main thm 1,main thm 2} show that $\{\pm1\}$-symmetry and $C_{4}$-symmetry are enough to prove parts of \cite[Theorem~1.1]{B} and \cite[Theorem~1.2]{B} respectively.
\end{rem}

\section{Proof of \cref{main appl 1,main appl 2}}
\label{section: Proof of main appl 1}

In this \lcnamecref{section: Proof of main appl 1} we give the proofs of \cref{main appl 1,main appl 2}.
For this purpose, we first collect some preliminary results.
Let $X$ be an oriented connected closed smooth $4$-manifold with a double cover $\tilde{X}\to X$.
The following lemma is given in \cite{N1}.
(See \cite[Proposition~11]{N1} and the proof of \cite[Theorem~37]{N1}.)

\begin{lem}[\cite{N1}]
\label{lem: suff cond spinc-}
For each cohomology class $C \in H^{2}(X ;\ell)$,
let $[C]_{2} \in H^{2}(X;\Z_{2})$ denote the mod 2 reduction of $C$.
If $[C]_{2}$ satisfies
\[
[C]_{2} = w_{2}(X) + w_{1}(\ell_{\R})^{2},
\]
then there exists a $\Spincm$-structure $\mathfrak{s}$ on $\tilde{X}\to X$ such that $\tilde{c}_1(\fraks)=C$.
\end{lem}

Note that, as well as usual $\rm{Spin}^{c}$ structure, we may define the notion of a {\it topological $\Spincm$-structure} on a topological manifold and a {\it families topological $\Spincm$-structure} on a continuous bundle of manifolds, namely a manifold bundle whose structure group is the homeomorphism group of the fiber.
(See \cite[Subsection~4.2]{BK} for (families) topological $\rm{Spin}^{c}$ structures.)
Given a continuous bundle of manifolds and a families topological $\Spincm$-structure on it,
if the manifold bundle is smoothable, then the families topological $\Spincm$-structure induces a families $\Spincm$-structure in the usual sense.

\begin{lem}
\label{lem: reduction to Aut connected sum}
For $i=1,\ldots, n$, let $X_{i}$ be an oriented closed $4$-manifold, $\tilde{X}_{i} \to X_{i}$ be a double cover, $\fraks_{i}$ be a $\Spincm$-structure on $\tilde{X}_{i} \to X_{i}$, 
$f_{i}$ be a self-diffeomorphism of $X_{i}$ preserving orientation of $X_{i}$ and the isomorphism class of $\fraks_{i}$.
Suppose that each $f_{i}$ has a fixed ball $B_{i}$ embedded in $X_{i}$, and extend $f_{i}$ to a self-diffeomorphism of $X$ by identity outside $X_{i}$.
Define the connected sums $X=X_{1} \# \cdots \#X_{n}$ and $\fraks = \fraks_{1} \# \cdots \# \fraks_{n}$ gluing around $B_{i}$.
Then there exist commuting lifts $\tilde{f}_{1}, \ldots, \tilde{f}_{n}$ in $\Aut(X,\fraks)$ of the commuting diffeomorphisms $f_{1}, \ldots, f_{n}$.

Moreover, a similar statement holds also for topological $\Spincm$-structures.
\end{lem}

\begin{proof}
The proof of the case for topological $\Spincm$-structures is similar to the smooth case, so we give the proof only for the smooth case.
Note that we have an exact sequence
\[
1 \to \gauge(X) \to \Aut(X,\fraks) \to \Diff(X, [\fraks]) \to 1,
\]
where $\gauge(X)$ is the gauge group of the $\Spincm$-structure $\fraks$ and $\Diff(X, [\fraks])$ is the group of diffeomorphisms preserving the isomorphism class of $\fraks$.
Take a lift $\hat{f}_{i}$ in $\Aut(X,\fraks)$ of $f_{i}$.
Since $f_{i}$ is supported inside $X_{i} \setminus B_{i}$, we have that
\[
\hat{f}_{i}|_{X \setminus (X_{i} \setminus B_{i})} \in \gauge(X \setminus (X_{i} \setminus B_{i})).
\]
Set $u_{i} = \hat{f}_{i}|_{X \setminus (X_{i} \setminus B_{i})}$.
To complete the proof of the \lcnamecref{lem: reduction to Aut connected sum},
it suffices to show that there exists an extension of each $u_{i}$ to an element of $\gauge(X)$, since then the lifts $\tilde{f}_{i} := u_{i}^{-1} \cdot \hat{f}_{i}$ of $f_{i}$ satisfy the desired property.

To see that $u_{i} \in \gauge(X \setminus (X_{i} \setminus B_{i}))$ can be extended to an element of $\gauge(X)$,
note that we may assume that $\tilde{X}_{i} \to X_{i}$ is the trivial double cover around $B_{i}$ and that $\fraks$ is a trivial $\Spincm$-structure around $B_{i}$.
Then, as noted in \cite[Remark~2.8]{N2},
we may regard $u_{i}|_{\partial B_{i}}$ as a map $u_{i}|_{\partial B_{i}} : S^{3} \to U(1)$, which can be deformed continuously to the constant map onto the identity element in $U(1)$ since $\pi_{3}(U(1))=0$.
This implies that $u_{i}$ can be extended as we desired.
\end{proof}

We can now start the proof of \cref{main appl 1}.
Some of ideas of the construction of a non-smoothable family $E$ with fiber $X$ are based on \cite[Section~2]{N1}, \cite[Section~1]{N2}, \cite[Sections~3, 4]{N0}, \cite[Theorem~4.1]{KKN}, and \cite[Theorem 10.3]{B}.

\begin{proof}[Proof of \cref{main appl 1}]
Let $X$ be as in the statement of \cref{main appl 1}.
Set
\begin{align}
\label{eq: def of N}
&N = \#_{i=1}^{p}(S^{1} \times Y_{i})\#_{j=1}^{q} (S^{2} \times \Si_{j}).
\end{align}
Since the assertion of \cref{main appl 1} is invariant under reversing orientation of $M$, we may assume that $\sigma(M)<0$ without loss of generality.
Then we have $n=b_{+}(M)$.
Note that, since $M$ is assumed to be indefinite, we have $b_{+}(M)>0$.

Recall that the double covers of $N$ are classified by 
\begin{align}
\label{eq: H1N decom}
H^{1}(N;\Z_{2}) &\cong \bigoplus_{i} H^{1}(S^{1} \times Y_{i};\Z_{2}) \bigoplus_{j} H^{1}(S^{2} \times \Si_{j};\Z_{2}).
\end{align}
Let $\tilde{N} \to N$ be the double cover of $N$ corresponding to a cohomology class in $H^{1}(N;\Z_{2})$ whose image under the projection onto each of the direct summands under the decomposition \eqref{eq: H1N decom} does not zero.
Set $\ell^{N} = \tilde{N} \times_{\pm1} \Z$ and
$\ell^{N}_{\R} = \tilde{N} \times_{\pm1} \R$.
Then it follows that 
\begin{align}
\label{eq: properties of N}
b_{2}^{\ell^{N}}(N) = 0, \quad{\rm and }\quad w(\ell^{N}_{\R})^{2}=0.
\end{align}
Let $\tilde{X} \to X$ be the fiberwise connected sum of the trivial double cover $M \sqcup M \to M$ and $\tilde{N} \to N$.
Set $\ell=\tilde{X} \times_{\pm1}\Z$ and $\ell_{\R}=\tilde{X} \times_{\pm1}\R$.
Then we have
\begin{align}
\label{eq: HXMN}
H^{2}(X;\ell) \cong H^{2}(M;\Z) \oplus H^{2}(N;\ell^{N})
\end{align}
and
\begin{align}
\label{eq: comparison w1}
w_{1}(\ell_{\R})^{2} = (0,w_{1}(\ell^{N}_{\R})^{2})
\end{align}
through \eqref{eq: HXMN}, and also have
\[
b_{+}^{\ell}(X)=b_{+}(M)=n.
\]
It follows from \eqref{eq: properties of N} and \eqref{eq: comparison w1} that
\begin{align}
\label{eq: rewritten the assumption}
w_{2}(X) + w_{1}(\ell_{\R})^{2} = w_{2}(M)
\end{align}
since we have $w_{2}(N)=0$.
Below we consider the case that $M$ is spin and that $M$ is non-spin separately.

First, let us consider the case that $M$ is spin.
In this case, $M$ is homeomorphic to 
\begin{align}
\label{eq: top mfd spin}
2m(-E_{8}) \# nS^{2} \times S^{2}
\end{align}
for some $m$ by Freedman's theory, where $-E_{8}$ denotes the negative-definite $E_{8}$-manifold.
Note that we have $m>0$ since we have assumed that $\sigma(M)<0$ (actually we also have $n \geq 2m+1$ by Furuta's 10/8-inequality, but this fact is not necessary here).
Henceforth we shall identify $M$ with \eqref{eq: top mfd spin} as topological manifold.

As noted in \cite[Example~3.3]{N0}, one may easily find an orientation-preserving self-diffeomorphism $\varrho : S^2\times S^2 \to S^2\times S^2$ satisfying the following two properties:
\begin{itemize}
\item There exists a $4$-ball $B$ embedded in $S^2\times S^2$ such that the restriction of $\varrho$ on a neighborhood of $B$ is the identity map.
\item $\varrho$ reverses orientation of $H^+(S^2\times S^2)$.
\end{itemize}
Let $f_{1}, \ldots, f_{n-1}$ be copies of $\varrho$ on each connected summand of $(n-1)(S^{2} \times S^{2})$, and let us extend them as homeomorphisms of $M$ and $X$ by identity over the other connected sum factors.
Since $f_{1}, \ldots, f_{n-1}$ commute with each other, we can form the multiple mapping torus
\[
X \to E \to T^{n-1}
\]
of $f_{1}, \ldots, f_{n-1}$.
This family $E$ is a $\Homeo(X)$-bundle, for which we shall show non-smoothability.
We argue by contradiction and suppose that the family $X \to E \to T^{n-1}$ has a reduction of structure group to $\Diff(X)$.

Let $M \to E_{M} \to T^{n-1}$ denote the multiple mapping torus of $f_{1}, \ldots, f_{n-1}$ restricted to $M$.
Then the family $E$ is the fiberwise connected sum of $E_{M}$ and the trivialized bundle $T^{n-1} \times N \to N$.
As in the proof of \cite[Theorem~10.3]{B}, it is easy to see that $w_{n-1}(H^{+}(E_{M})) \neq 0$.
This non-vanishing together with \eqref{eq: properties of N} and \eqref{eq: HXMN} implies that 
\begin{align}
\label{eq: non-vanishing in main app}
w_{n-1}(H^+(E,\ell))\neq 0 \text{ in } H^{n-1}(B;\Z_{2}).
\end{align}

Since now we have $w_{2}(M)=0$, it follows from \cref{lem: suff cond spinc-} and the equation \eqref{eq: rewritten the assumption} that there exists a $\Spincm$-structure $\mathfrak{s}$ on $\tilde{X}\to X$ such that $\tilde{c}_1(\fraks)=0$.
More precisely, we may take $\mathfrak{s}$ to be trivial on the conneced summand $M$ in $X$.
Here we note the following \lcnamecref{lem: reduction to Aut}:

\begin{lem}
\label{lem: reduction to Aut}
The family $E$ has a reduction  of structure group to $\Aut(X,\mathfrak{s})$, provided that
$E$ has a reduction of structure group to $\Diff(X)$.
\end{lem}

\begin{proof}
Since the $\Spincm$-structure $\fraks$ on the conneced summand $M$ in $X$ is trivial,
each $f_{i}$ obviously preserves the isomorphism class of the resrtriction of the topological $\Spincm$-structure $\fraks$ on the $i$-th conneced summand of $n(S^{2} \times S^{2})$.
Therefore this \lcnamecref{lem: reduction to Aut} follows from
\cref{lem: reduction to Aut connected sum}.
\end{proof}

We can now complete the proof of \cref{main appl 1} in the case that $M$ is spin.
By \eqref{eq: non-vanishing in main app} and \cref{lem: reduction to Aut},
the family $X \to E \to T^{n-1}$ satisfies the assumption of \cref{main thm 2}, thus we have $\sigma(X) \geq 0$.
However $\sigma(X) = \sigma(M)$ holds and we assumed that $\sigma(M)<0$.
This is a contradiction, and hence $E$ is non-smoothable.

Next, let us consider the case that $M$ is not spin.
Some of arguments here are very similar to the spin case above.
Denote by $-\CP^{2}_{\fake}$ the closed simply-connected topological $4$-manifold whose intersection form is $(-1)$ and whose Kirby--Siebenmann class does not vanish.
Then $M$ is homeomorphic to
\[
m(-\CP^{2}) \# (-E_{8}) \#(-\CP^{2}_{\fake})\# n(S^{2} \times S^{2})
\]
for some $m \geq 0$ and $n>0$.
Let $f_{1}, \ldots, f_{n}$ be the commuting self-diffeomorphisms of $n(S^{2} \times S^{2})$ obtained as copies of $\varrho$ above as well as the spin case, and extending them as self-homeomorphisms of $X$ by identity, we may obtain a continuous family $X \to E \to T^{n}$ as the multiple mapping torus.
Similar to the spin case, we argue by contradiction and suppose that the family $X \to E \to T^{n}$ has a reduction of structure group to $\Diff(X)$.

Let $M \to E_{M} \to T^{n}$ denote the multiple mapping torus of $f_{1}, \ldots, f_{n}$ restricted to $M$.
Then it is easy to see that $e(H^{+}(E_{M}, \Z_{w_{1}(H^{+}(E_{M}))})) \neq 0$, where $\Z_{w_{1}(H^{+}(E_{M}))}$ denotes the local system with coefficient group $\Z$ determined by $w_{1}(H^{+}(E_{M}))$.
This observation together with \eqref{eq: properties of N} and \eqref{eq: HXMN} implies that 
\begin{align}
\label{eq: non-vanishing euler in main app}
w_{n}(H^+(E,\ell))\neq 0 \text{ in } H^{n}(B;\Z_{2}).
\end{align}

Let $C \in H^{2}(X;\ell)$ be a cohomology class expressed as
\[
C = (e_{1}, \ldots, e_{m}, 0, e, 0,0)
\]
under the direct sum decomposition of $H^{2}(X;\ell)$ into
\[
H^{2}(-\CP^{2};\Z)^{\oplus m} \oplus H^{2}(-E_{8};\Z) \oplus H^{2}(-\CP^{2}_{\fake};\Z) \oplus H^{2}(n(S^{2} \times S^{2});\Z) \oplus H^{2}(N;\ell^{N}),
\]
where $e_{i}$ and $e$ denote a generator of $H^{2}(-\CP^{2};\Z)$ and that of $H^{2}(-\CP^{2}_{\fake};\Z)$ respectively.
Then $C$ satisfies that $[C]_{2} = w_{2}(M)$.
Therefore it follows from \cref{lem: suff cond spinc-} and \eqref{eq: rewritten the assumption} that 
there exists a $\Spincm$-structure $\mathfrak{s}$ on $\tilde{X}\to X$ such that $\tilde{c}_1(\fraks)=C$.

As well as \cref{lem: reduction to Aut}, the structure group of $E$ lifts to $\Aut(X,\fraks)$ provided that $E$ is smoothable.
Therefore by \eqref{eq: non-vanishing euler in main app} we may apply \cref{main thm 1} to this family, and thus we have $\tilde{c}_1(\fraks)^2 \leq \sigma(X)$.
However it follows from a direct calculation that $\tilde{c}_1(\fraks)^2 = C^{2} = -m-1$ and $\sigma(X) = \sigma(M) = -m-9$.
This is a contradiction, and hence $E$ is non-smoothable.
This completes the proof of \cref{main appl 1}.
\end{proof}

\begin{proof}[Proof of \cref{main appl 2}]
The proof is very similar to that of \cref{main appl 1} above.
Let $X$ be as in the statement of \cref{main appl 2} and $M'=mS\#M$.
Define $N$ by \eqref{eq: def of N}.
Recall that an Enriques surface $S$ is homeomorphic to $-E_{8} \# (S^{2} \times S^{2}) \# W$, where $W$ is a non-spin topological rational homology $4$-sphere with $\pi_{1}(W) \cong \Z/2$ and with non-trivial Kirby--Siebenmann invariant.
Hence $mS$ is homeomorphic to 
\[
m(-E_{8}) \# mS^{2} \times S^{2} \# mW.
\]
If $M$ is non-spin, then $M$ is homemorphic to 
$a\CP^{2}\#b(-\CP^{2})$ for some $a,b \geq 0$, and 
if $M$ is spin, then $M$ is homemorphic to
$a(S^{2} \times S^{2})\# 2b(-E_{8}) \# $
for some $a,b \geq 0$, since we assumed $\sigma(M) \leq 0$ in the spin case.
Let us repeat the argument in the proof of \cref{main appl 1} until getting the equation \eqref{eq: rewritten the assumption} under replacing $M$ with $M'$.

First, let us assume that $M$ is spin.
Then $M'$ is homeomorphic to
\[
(m+2b)(-E_{8})\#nS^{2} \times S^{2} \#mW,
\]
where $n=a+m$.
Let $f_{1}, \ldots, f_{n}$ be the commuting self-diffeomorphisms of $n(S^{2} \times S^{2})$ obtained as copies of $\varrho$ given in the proof of \cref{main appl 1}, and extending them as self-homeomorphisms of $X$ by identity, we may obtain a continuous family $X \to E \to T^{n}$ as the multiple mapping torus.
We argue by contradiction and suppose that the family $X \to E \to T^{n}$ has a reduction of structure group to $\Diff(X)$.
First, note that we again obtain \eqref{eq: non-vanishing euler in main app} similarly.
Let $\alpha \in H^{2}(S;\Z)$  be the cohomology class given by
$\alpha = (0,1) \in H^{2}(S;\Z)$ under the direct sum decomposition
\[
H^{2}(S;\Z) \cong H^{2}(-E_{8} \# S^{2} \times S^{2};\Z) \oplus H^{2}(W;\Z),
\]
where $H^{2}(W;\Z)$ is known to be isomorphic to $\Z/2\Z$ and $1 \in H^{2}(W;\Z)$ denotes the unique non-trivial element.
Let $C \in H^{2}(X;\ell)$ be the cohomology class given by
\[
C = (0, \alpha_{1}, \ldots, \alpha_{m}, 0)
\]
under the decomposition of $H^{2}(X;\ell)$ into
\[
H^{2}((m+2b)(-E_{8}) \# nS^{2} \times S^{2};\Z) \oplus H^{2}(W;\Z)^{\oplus m} \oplus H^{2}(N;\ell^{N}),
\]
where $\alpha_{i}$ are copies of $\alpha$.
Then $C$ satisfies that $[C]_{2} = w_{2}(M')$.
Then we can deduce from an argument similar to the proof of \cref{main appl 1} that $C^{2} \leq \sigma(X)$ using \cref{main thm 1}.
However it follows from a direct calculation that $C^{2}=0$ and $\sigma(X) = -8(m+2b)$.
This is a contradiction, and hence $E$ is non-smoothable.
This completes the proof of \cref{main appl 2} in the spin case.

Next, let us assume that $M$ is non-spin.
The proof is similar to the spin case above.
First, note that $M'$ is homeomorphic to
\[
m(-E_{8})\#n\CP^{2} \# n'(-\CP^{2}) \#mW,
\]
where $n=a+m$ and $n'=b+m$.
Let $\rho$ be an orientation-preserving self-diffeomorphism of $\CP^{2}$ satisfying the following two properties:
\begin{itemize}
\item There exists a $4$-ball $B$ embedded in $\CP^{2}$ such that the restriction of $\varrho$ on a neighborhood of $B$ is the identity map.
\item $\varrho$ reverses orientation of $H^+(\CP^{2})$.
\end{itemize}
One may get an example of such $\rho$ as follows: let $\rho' : \CP^{2} \to \CP^{2}$ the complex conjugation $[z_{0}:z_{1}:z_{2}] \mapsto [\bar{z}_{0}:\bar{z}_{1}:\bar{z}_{2}]$.
Take a point from the fixed point set $\RP^{2} \subset \CP^{2}$ of $\rho'$, and deform $\rho'$ by isotopy around the point to obtain a fixed ball $B$.
This deformed self-diffeomorphism $\rho$ satisfies the desired conditions.
Let $f_{1}, \ldots, f_{n}$ be the commuting self-diffeomorphisms of $n\CP^{2}$ obtained as copies of $\varrho$, and extending them as self-homeomorphisms of $X$ by identity, we may obtain a continuous family $X \to E \to T^{n}$ from $f_{1}, \ldots, f_{n}$ as well.
Suppose that the family $X \to E \to T^{n}$ has a reduction of structure group to $\Diff(X)$.
We again obtain \eqref{eq: non-vanishing euler in main app} similarly.
Let $e$ and $\bar{e}$ are generators of $H^{2}(\CP^{2};\Z)$ and $H^{2}(-\CP^{2};\Z)$ respectively.
Let $C \in H^{2}(X;\ell)$ be the cohomology class given by
\[
C = (0, e_{1}, \ldots, e_{n}, \bar{e}_{1}, \ldots, \bar{e}_{n'},\alpha_{1}, \ldots, \alpha_{m}, 0)
\]
under the decomposition of $H^{2}(X;\ell)$ into
\[
H^{2}(m(-E_{8});\Z) \# H^{2}(n\CP^{2};\Z)\# H^{2}(n'(-\CP^{2});\Z)
\oplus H^{2}(W;\Z)^{\oplus m} \oplus H^{2}(N;\ell^{N}),
\]
where $e_{i}$ and $\bar{e}_{j}$ are copies of $e$ and $\bar{e}$ respectively.
Then $C$ satisfies that $[C]_{2} = w_{2}(M')$, and
we can deduce that $C^{2} \leq \sigma(X)$ using \cref{main thm 1}.
However it follows from a direct calculation that $C^{2}=n-n'$ and $\sigma(X) = -8m+n-n'$.
This is a contradiction, and hence $E$ is non-smoothable.
This completes the proof of \cref{main appl 2} in the non-spin case.
\end{proof}

\end{document}